\documentclass[11pt]{article}

\usepackage{latexsym,color,graphicx,amsmath,amssymb,amsthm}

\def\R{{\mathbb R}}

\def\C{{\mathbb C}}

\def\u{{\bf u}}

\def\Log{\text{Log}}
\def\Z{\mathbb{Z}}

\newtheorem{thm}{Theorem}[section]
\newtheorem{cor}[thm]{Corollary}
\newtheorem{lem}[thm]{Lemma}
\newtheorem{clm}[thm]{Claim}

\theoremstyle{definition}

\begin{document}

\title{Expanding polynomials: A generalization of the Elekes-R\'onyai theorem to $d$ variables}

\author{
Orit E. Raz\thanks{%
Mathematics Department, University of British Columbia,
Vancouver BC, Canada.
{\sl oritraz@math.ubc.ca} }
\and
Zvi Shem Tov\thanks{%
Mathematics Department, University of British Columbia,
Vancouver BC, Canada.
{\sl zvishem@math.ubc.ca} }
 }

\maketitle

\abstract{We prove the following statement.
Let $f\in\R[x_1,\ldots,x_d]$, for some $d\ge 3$, and assume that $f$ depends non-trivially in each of $x_1,\ldots,x_d$.
Then one of the following holds.\\
(i) For every finite sets $A_1,\ldots,A_d\subset \R$, each of size $n$, we have
$$|f(A_1\times\ldots\times A_d)|=\Omega(n^{3/2}),
$$
with constant of proportionality that depends on $\deg f$.\\
(ii) $f$ is of one of the forms
\begin{align*}
f(x_1,\ldots, x_d)&=h(p_1(x_1)+\cdots+p_d(x_d))~~\text{or}\\
f(x_1,\ldots, x_d)&=h(p_1(x_1)\cdot\ldots\cdot p_d(x_d)),
\end{align*}
for some univariate real polynomials $h(x)$, $p_1(x),\ldots,p_d(x)$.
This generalizes the results from \cite{ER00,RSS1, RSdZ2}, which treat the cases $d=2$ and $d=3$.
}
\section{Introduction}

Let $A,B$ be two sets, each of $n$ real numbers, and let $f$ be a real bivariate polynomial of constant degree. Elekes and R\'onyai~\cite{ER00} showed that if $|f(A\times B)|\le cn$, for some constant $c$ that depends only on $\deg f$, and for $n\ge n_0(c)$, for sufficiently large threshold $n_0(c)$ that depends on $c$, then $f$ must be of one of the special forms $f(x,y)=h(p(x)+q(y))$, or $f(x,y)=h(p(x)\cdot q(y))$, for some univariate polynomials $p,q,h$ over $\R$. Later, Raz, Sharir, and Solymosi~\cite{RSS1} showed that $|f(A\times B)|=\Omega(n^{4/3})$, for every $A,B$ each of size $n$ (with constant of proportionality that depends only on the degree of $f$),  unless $f$ has one of the above mentioned special forms.  
This result was extended by Raz, Sharir, and De Zeeuw in \cite{RSdZ2} to polynomials of three variables. 
Concretely, it is proved in \cite{RSdZ2} that, for a real trivariate polynomial $f$,  either
$|f(A\times B\times C)|=\Omega(n^{3/2})$ for every $A,B,C\subset \R$ each of size $n$, or $f$ is of one of the special forms $f(x,y,z)=h(p(x)+q(y)+r(z))$ or $f(x,y,z)=h(p(x)\cdot q(y)\cdot r(z))$.

The main result of this paper is extending this line of results to polynomials with arbitrary number of variables. We state a combined version for polynomials over $\C$ and over $\R$.

\begin{thm}\label{main}
Let $f\in\C[x_1,\ldots,x_d]$, for some $d\ge 3$, and assume that $f$ depends non-trivially in each of $x_1,\ldots,x_d$.
Then one of the following holds.\\
(i) For every finite sets $A_1,\ldots,A_d\subset \C$, each of size $n$, we have
$$|f(A_1\times\ldots\times A_d)|=\Omega(n^{3/2}),
$$
with constant of proportionality that depends on $\deg f$.\\
(ii) $f$ is of one of the forms
\begin{align*}
f(x_1,\ldots, x_d)&=h(p_1(x_1)+\cdots+p_d(x_d))~~\text{or}\\
f(x_1,\ldots, x_d)&=h(p_1(x_1)\cdot\ldots\cdot p_d(x_d)),
\end{align*}
for some univariate complex polynomials $h(x)$, $p_1(x),\ldots,p_d(x)$. Moreover, if $f$ is a real polynomial, then there exist $h(x)$, $p_1(x),\ldots,p_d(x)$ with real coefficients.
\end{thm}

\vspace{.5cm}
\noindent {\it Remark.} A related problem, studied in \cite{ESz, RSdZ1}, deals with algebraic varieties of the form $F(x,y,z)=0$ or $F(x,y,z,w)=0$ in $\C^4$ (not restricted to be graphs of polynomial functions $z=f(x,y)$ or $w=f(x,y,z)$ as in our setup). A recent paper by Bays and Breuillard~\cite{BB18} studies varieties $V$ in $\C^k$ that {\it admits no power-saving}:  Roughly, this means that for every $n$, there exist $A_1,\ldots,A_d\subset \C$, each of size $n$, such that $|V\cap (A_1\times \cdots\times A_d)|=\Theta^*(n^{\dim V})$, where the $*$ in the $O$-notation stands for any sub-polynomial factor (see \cite{BB18} for the precise definition).
Note that in case the alternative (ii) in Theorem~\ref{main} holds for a given $d$-variate polynomial $f$, then the $d$-dimensional variety $V:=\{y=f(x_1,\ldots,x_d)\}\subset \C^{d+1}$ admits no power-saving.
Thus, the result in \cite{BB18} provides a description of such $V$. However, this description is somewhat weaker and less concrete than ours, as it applies to a more general setup.

\paragraph{Sketch of Proof.}
In our analysis we show that the following phenomenon occurs.
Let  $f\in \C[x_1,\ldots,x_k]$. For every triple $\{i_1,i_2,i_3\}\subset \{1,\ldots,k\}$, consider the trivariate polynomial induced by $f$ by setting certain generic values to $\{x_1,\ldots, x_k\}\setminus\{i_1,i_2,i_3\}$. 
Suppose that each of these induced trivariate polynomials is special (in the sense of \cite{RSdZ2} and as stated in the introduction). In this case we show that $f$ is special as a $d$-variate polynomial; i.e., property (ii) in Theorem~\ref{main} holds for $f$. 
On the other hand, if one of these induced trivariate polynomials turns out not to be special, then by applying (essentially) the result for trivariate polynomials, we show that property (i) in Theorem~\ref{main} holds for $f$.

In more detail, in Lemma~\ref{analytic3vars} we extend the results from \cite{RSdZ2}, to handle the case of trivariate polynomials, with coefficients  that are themselves polynomials of $d-3$ variables. We prove that substituting generic values to the coefficients results in either a trivariate polynomial which is not special, in which case the expansion property (Lemma~\ref{analytic3vars}(i')) holds, or certain differential equations hold (Lemma~\ref{analytic3vars}(ii')).
In the latter case, the differential equations imply that the relevant trivariate polynomial is special, but we do not elaborate this, since we first want to obtain a system of differential equations that is more symmetric in the $d$-variables of $f$, from which we will be able to deduce one of the forms in Theorem~\ref{main}(ii).

In case restricting $f$ to each triple of variables induces (generically) a trivariate polynomial that is special, the differential equations we get from  
Lemma~\ref{analytic3vars}(ii') will allow us to deduce a more symmetric system of differential equations; this is obtained in Lemma~\ref{3tod}.
It is not hard to show (for instance, by generalizing the analysis of \cite{RSdZ2}), that this system of equations implies a ``local'' version of Theorem~\ref{main}(ii), where $x_1,\ldots,x_d$ are taken from an open neighborhood of $\C^d$ and $p_1,\ldots,p_d,h$ are univariate analytic functions, defined on an open subset of $\C$.
To show the stronger statement, that $p_1,\ldots,p_d,h$ can be taken to be univariate polynomials, we apply the analysis of Tao~\cite{Tao}. In \cite{Tao} the bivariate case is considered, and we generalize Tao's analysis to $d$ variables in Lemma~\ref{speform}. 
For completeness we provide full details of this step, even if it in parts overlaps the proof of Tao. In addition, we provide the background needed from complex analysis in the Appendix. 

\vspace{.5cm}
\noindent {\it Remark.} Consider the example $f(x,y,z)=xy+z$. Setting value to any of $x,y,z$ results in a bivariate polynomial, which has one of the special forms, in the sense of \cite{ER00,RSS1}. Nevertheless, $f$ is not of one of the special forms as a trivariate polynomial. This example shows that, in order to determine whether a trivariate polynomial is special, it is not sufficient to consider the bivariate polynomials induced by $f$ by fixing one of the variables. This means that the approach taken in this paper is applicable only to the case of $d\ge 4$ variables.

\section{Proof of main result}
\subsection{Three main lemmas}
\label{sec:ext}
We split the proof of our main result, Theorem~\ref{main}, into three lemmas.

In Lemma~\ref{analytic3vars}, we extend the analysis of Raz, Sharir, and De Zeeuw~\cite{RSdZ2} from trivariate to $d$-variate polynomial functions, regarding $d-3$ of the variables as fixed parameters. We prove the following dichotomy that generalizes the one from~\cite{RSdZ2}; the proof of the lemma is given in Section~\ref{sec:lemRSdZ}.

\begin{lem}\label{analytic3vars}
Let $f\in\C[x,y,z,u_1,\ldots,u_l]$, for some $l\ge 1$.
Then one of the following holds.\\
(i')  There exists 
a constant-degree algebraic subvariety $\mathcal U\subsetneq\C^l$
such that, for every $\u\in\C^l\setminus\mathcal U$ and every finite sets $A,B,C\subset \C$, we have
$$|f(A\times B\times C\times \{\u\})|=\Omega\left(\min\left\{(|A||B||C|)^{1/2},|A||B|\right\}\right).
$$
(ii') We have the identity
$$
\frac{\tfrac{\partial f}{\partial x}(x,y,z,\u)}{p(x,\u)}
=\frac{\tfrac{\partial f}{\partial y}(x,y,z,\u)}{q(y,\u)}
=\frac{\tfrac{\partial f}{\partial z}(x,y,z,\u)}{r(z,\u)},
$$
for some rational functions $p,q,r$, each of $l+1$ variables.
\end{lem}

In Lemma~\ref{3tod} we show how to use the previous lemma in order to deduce  a stronger differential equation that involves all the variables. This lemma is the main technical part, where we go from trivariate to $d$-variate polynomial functions.
The proof of the lemma is given in Section~\ref{sec:3tod}.

\begin{lem}\label{3tod}
Let $f\in \C[x_1,\ldots,x_d]$ be a real polynomial of $n$ variables.
Assume that for every permutation $\sigma$ of $\{1,\ldots,d\}$, we have
$$
\frac
{\tfrac{\partial f}{\partial x_{\sigma(1)}}(x_1,\ldots,x_d)}
{r_{\sigma,1}(x_{\sigma(1)},x_{\sigma(4)},\ldots,x_{\sigma(d)})}
=
\frac
{\tfrac{\partial f}{\partial x_{\sigma(2)}}(x_1,\ldots,x_d)}
{r_{\sigma,2}(x_{\sigma(2)},x_{\sigma(4)},\ldots,x_{\sigma(d)})}
=
\frac
{\tfrac{\partial f}{\partial x_{\sigma(3)}}(x_1,\ldots,x_d)}
{r_{\sigma,3}(x_{\sigma(3)},x_{\sigma(4)},\ldots,x_{\sigma(d)})},
$$
for some $(d-2)$-variate rational functions $r_{\sigma,i}$.
Then there exist univariate rational functions $r_1,\ldots,r_d$, such that 
$$
\frac
{\tfrac{\partial f}{\partial x_1}(x_1,\ldots,x_d)}
{r_1(x_1)}
=\cdots=
\frac
{\tfrac{\partial f}{\partial x_d}(x_1,\ldots,x_d)}
{r_d(x_d)}
$$
\end{lem}

Finally, in Lemma~\ref{speform}, we prove that the system of differential equations \eqref{de} (deduced in Lemma~\ref{3tod}) implies that $f$ has one of the forms specified in Theorem~\ref{main} property (ii).
The proof is an extension of the analysis of Tao~\cite{Tao} for polynomials of two variables; the details are given in Section~\ref{sec:speform}.
\begin{lem}\label{speform}
Let $f\in \C[x_1,\ldots,x_d]$ be a complex polynomial of $d$ variables.
Assume that 
\begin{equation}\label{de}
\frac
{\tfrac{\partial f}{\partial x_1}(x_1,\ldots,x_d)}
{r_1(x_1)}
=\cdots=
\frac
{\tfrac{\partial f}{\partial x_d}(x_1,\ldots,x_d)}
{r_d(x_d)}
\end{equation}
for some univariate rational functions $r_1,\ldots,r_d$.
Then $f$ is of one of the forms
\begin{align*}
f(x_1,\ldots, x_d)&=h(p_1(x_1)+\cdots+p_d(x_d))~~\text{or}\\
f(x_1,\ldots, x_d)&=h(p_1(x_1)\cdot\ldots\cdot p_d(x_d)),
\end{align*}
for some univariate polynomials $h(x)$, $p_1(x),\ldots,p_d(x)$.
Moreover, if $f$ is a polynomial with real coefficients,
then $h,p_1,\ldots,p_d$ in the conclusion can be taken to be {\em real} univariate polynomials.
\end{lem}

\subsection{Proof of Theorem~\ref{main}} 
We now prove Theorem~\ref{main}, given Lemmas~\ref{analytic3vars}, \ref{3tod} and \ref{speform}.
Let $f\in \C[x_1,\ldots,x_d]$ be as in the statement. 
For any permutation $\sigma$ of $\{1,\ldots,d\}$ 
we rename the coordinates $(x_{\sigma(1)},\ldots,x_{\sigma(d)})$ as $(x,y,z,u_1,\ldots,u_l)$, where $l=d-3$, and apply Lemma~\ref{analytic3vars}.

Assume first that for {\it some} permutation $\sigma$, property (i') of Lemma~\ref{analytic3vars} holds. In this case we show that property (i) of Theorem~\ref{main} holds. Indeed, assume that, for a given permutation $\sigma$ and the corresponding renaming of the variables as $(x,y,z,u_1,\ldots,u_l)$, property (i') of Lemma~\ref{analytic3vars} holds. Let $\mathcal U$ be the variety from the statement of property (i'). 
Consider any finite sets $A,B,C,U_1,\ldots,U_l\subset\C$, each of size $n$.
Since $\mathcal U\subset \C^l$ is of codimension at least 1 and of constant degree, then, for $n$ large enough, $U_1\times\cdots\times U_l\not\subseteq \mathcal U$.
Thus, there exists $\u_0\in(U_1\times\cdots\times U_l)\cap(\C^l\setminus \mathcal U)$.
By property (i'), we have 
$$|f(A\times B\times C\times \{\u_0\})|=\Omega(n^{3/2})
$$
and thus clearly
$$|f(A\times B\times C\times U_1\times\cdots\times U_l)|=\Omega(n^{3/2}).
$$
So in this case property (i) of Theorem~\ref{main} holds, and we are done.

Assume next that, for {\it every} permutation of the variables, property (ii') of Lemma~\ref{analytic3vars} holds. In this case the assumptions, and hence also the conclusion, of Lemmas~\ref{3tod} hold. Combining this with Lemma~\ref{speform} proves property (ii) of Theorem~\ref{main}. This completes the proof of the theorem. \hfill $\qed$

\section{Proof of Lemma~\ref{analytic3vars}}\label{sec:lemRSdZ}
\subsection{Review of results from \cite{RSdZ2}}\label{review}
In this section we restate results proved in \cite{RSdZ2}.
Let $F\in \C[x,y,z,w]$ be an irreducible 4-variate polynomial, and assume that $F$ depends non-trivially in each of its variables.
Following \cite[Section~3]{RSdZ2}, we define varieties $V$, $W$, and $\widetilde W$. Define the variety in $\C^6$ 
$$
V := \{(x, y, x',y',s,t)\in\C^6\mid F(x,y,s,t)=0,F(x',y',s,t)=0\}
$$
By \cite[Lemma 3.1]{RSdZ2}, $V$ is 4-dimensional.
 Let $G\in \C[x,y,x',y',s,t]$ be the polynomial given by
$$
G=\tfrac{\partial F}{\partial s}(x,y,s,t)\tfrac{\partial F}{\partial t}(x',y',s,t)-\tfrac{\partial F}{\partial s}(x',y',s,t)\tfrac{\partial F}{\partial t}(x,y,s,t).
$$
Consider the subvariety $W := V\cap Z(G)$ of $V$. 

The variety $W$ might be 4-dimensional from the following ``trivial" reason.
Define 
$$
\mathcal{T}:=\{(c,d)\in\C^2\mid F(x,y,c,d)\equiv 0~\text{(as a polynomial in x and y)}\}
$$
If $\mathcal{T}$ is nonempty, then $\C^4\times \mathcal{T}\subset W$.
By \cite[Lemma~2.1]{RSdZ2}, $\mathcal{T}$ is finite and has cardinality at most $(\deg F)^2$. 
Let $\widetilde{W}$ be any irreducible component of $W$, which is not a component of $\C^4\times \mathcal{T}$, and has maximal dimension (among these components).

The following is proven in \cite[Section 3.3]{RSdZ2}. 
\begin{thm}[{\bf Raz, Sharir, De Zeeuw}~\cite{RSdZ2}]\label{rsdzdimW}
Let $F\in \C[x,y,z,w]$ be an irreducible 4-variate polynomial, and assume that $F$ depends non-trivially in each of its variables. Let $V$, $\cal T$ and $W$ be as above. Let $\widetilde{W}$ be any irreducible component of $W$, which is not a component of $\C^4\times \mathcal{T}$ and has maximal dimension.
If $\dim \widetilde{W}\le 3$, then 
$$
|\{F=0\})\cap (A\times B\times C\times D)|=O(|A||B||C||D|+|A||B|+|A||C|+|A||D|+|B||C|+|B||D|+|C||D|),
$$
with constant of proportionality that depends on $\deg F$.
\end{thm}

\subsection{Proof of Lemma~\ref{analytic3vars}}
For every $\u=(u_1,\ldots,u_l)\in \C^l$, let 
$$F_{\u}(x,y,s,t):=y-f(x,s,t,\u).$$
Note that $F_{\u}$ is irreducible for every $\u\in \C^l$ fixed.
We want to apply Theorem~\ref{rsdzdimW} to $F_\u$. 
For this we consider the varieties $V$, $\cal T$, and $W$, introduced in the previous Subsection~\ref{review}, that correspond to our function $F_\u$. Put
$$
V_\u:=\{(x,y,x',y',s,t)\in \C^6\mid F_\u(x,y,s,t)=0,~F_\u(x',y',s,t)=0\}.
$$
Note that
\begin{align*}
V_\u
&=\{(x,y,x',y',s,t)\mid y=f(x,s,t,\u), y'=f(x',s,t,\u)\}\\
&=\{(x, f(x,s,t,\u), x', f(x',s,t,\u),s,t)\mid x,x',s,t\in\C\},
\end{align*}
which is clearly 4-dimensional and irreducible.
Next, put  
$$
\mathcal{T}:=\{(c,d)\in\C^2\mid y\equiv f(x,c,d,\u)~\text{(as polynomials in $x$ and $y$)}\}.
$$
Note that we have $\cal T=\emptyset$.
Finally,  put 
\begin{align*}
G_\u
&=\tfrac{\partial F_\u}{\partial s}(x,y,s,t)\tfrac{\partial F_\u}{\partial t}(x',y',s,t)-\tfrac{\partial F_\u}{\partial s}(x',y',s,t)\tfrac{\partial F_\u}{\partial t}(x,y,s,t)\\
&=\tfrac{\partial f}{\partial s}(x,s,t,\u)\tfrac{\partial f}{\partial t}(x',s,t,\u)-\tfrac{\partial f}{\partial s}(x',s,t,\u)\tfrac{\partial f}{\partial t}(x,s,t,\u)
\end{align*}
and define $W_\u:=V_\u\cap \{G_\u=0\}$.

Since $V_\u$ is irreducible and $\cal T=\emptyset$, 
we can apply Theorem~\ref{rsdzdimW} with $\widetilde{W_\u}=W_\u$.
Indeed, if $\dim W_\u=4$, then in fact $W_\u=V_\u$ and $W_\u$ has a unique irreducible component.
Otherwise, in case $\dim W_\u\le 3$, then every irreducible component of $W_\u$ is of dimension at most $3$, and hence $\dim \widetilde W_\u\le 3$ for any proper choice of $\widetilde W_u$.

Observe in addition that $\dim W_\u=4$ if and only if $G_\u\equiv 0$ for every $x,x',s,t$ (note that indeed $G_\u$ is independent of $y$ and $y'$ in our case and that here we regard $\u$ as fixed).
Define
$$
{\cal U}:=\{\u\in\C^l\mid G_\u\equiv 0~\text{(as a polynomial in $x,x',s,t$)}\}.
$$
Assume first that ${\cal U}\neq \C^l$ and let $\u\in\C^l\setminus \cal U$. 
Apply Theorem~\ref{rsdzdimW} to the function $F_\u$,
with $A,B,C\subset \C$ arbitrary finite sets and with $D:=F_\u(A\times B\times C)$.
By our choice of the set $D$, we have 
\begin{equation}\label{n3lower}
|\{F_\u=0\}\cap (A\times B\times C\times D)|= |A||B||C|.
\end{equation}
On the other hand, since $\u\not\in\cal U$, we have $\dim W_\u\le 3$, and thus the inequality in Theorem~\ref{rsdzdimW} holds for $F_\u$ and the sets $A$, $B$, $C$, and $D$. Combining this with \eqref{n3lower}, we get
$$|D|=|f(A\times B\times C\times \{\u\})|=\Omega\left(\min\left\{(|A||B||C|)^{1/2},|A||B|\right\}\right).
$$
In other words, the inequality in property (i') of Lemma~\ref{analytic3vars} holds for every $\u\not\in\cal U$.

Put
$$
G(x,x',s,t,\u):=G_\u(x,x',s,t).
$$ 
So $G$ is a complex polynomial of $l+4$ variables. Note that the degree of $G$ is bounded by some function of $\deg f$, and thus can be regarded as constant.
Write
$$
G_\u(x,x',s,t)=\sum_{0\le i+j+k+\ell\le\deg G}\alpha_{ijk\ell}(\u)x^i(x')^js^kt^\ell,
$$
for some constant-degree polynomials $\alpha_{ijk\ell}$ in the variables $u_1,\ldots,u_l$.
Note that, for $\u\in\C^l$ fixed, $G_\u(x,x',s,t)\equiv 0$ (as a polynomial in $x,x',s,t$) if and only if $\alpha_{ijk\ell}(\u)=0$ for every $0\le i+j+k+\ell\le\deg G$.
That is, we have
$$
{\cal U}=\{\u\in\C^l\mid \alpha_{ijk\ell}(\u)=0~\text{for every}~ 0\le i+j+k+\ell\le\deg G\}.
$$
So $\cal U$ is a non-trivial constant-degree algebraic subvariety in $\C^l$, unless $\alpha_{ijk\ell}(\u)\equiv0$ (as a polynomial in $u_1,\ldots,u_l$) for every $0\le i+j+k+\ell\le\deg G$.
In the latter case, we have $G(x,x',s,t,\u)\equiv 0$ (as a polynomial in $l+4$ variables).

We conclude that either 
\begin{equation}\label{G=0}
\tfrac{\partial f}{\partial s}(x,s,t,\u)\tfrac{\partial f}{\partial t}(x',s,t,\u)\equiv\tfrac{\partial f}{\partial s}(x',s,t,\u)\tfrac{\partial f}{\partial t}(x,s,t,\u),
\end{equation}
for every $x,x',s,t$, and every $\u=(u_1,\ldots,u_l)$,
or $\cal U$ is a constant-degree variety of codimension at least one, and then property (i') of Lemma~\ref{analytic3vars} holds.

We repeat the analysis for 
$$
F_\u':=y-f(s,x,t,\u)
$$
and
$$
F_\u'':=y-f(s,t,x,\u)
$$
(permuting the roles of $x$, $s$, and $t$).
In each case, we either conclude that property (i') of Lemma~\ref{analytic3vars} holds, or get a certain polynomial identity which is the analogue of \eqref{G=0}.

We summarize what we have shown so far in the following lemma.
\begin{lem}\label{summarize}
Either property (i') in Lemma~\ref{analytic3vars} holds, or
\begin{align*}
\tfrac{\partial f}{\partial y}(x,y,z,\u)\tfrac{\partial f}{\partial z}(x',y,z,\u)&\equiv\tfrac{\partial f}{\partial y}(x',y,z,\u)\tfrac{\partial f}{\partial z}(x,y,z,\u)\\
\tfrac{\partial f}{\partial x}(x,y,z,\u)\tfrac{\partial f}{\partial z}(x,y',z,\u)&\equiv\tfrac{\partial f}{\partial x}(x,y',z,\u)\tfrac{\partial f}{\partial z}(x,y,z,\u)\\
\tfrac{\partial f}{\partial x}(x,y,z,\u)\tfrac{\partial f}{\partial y}(x,y,z',\u)&\equiv\tfrac{\partial f}{\partial x}(x,y,z',\u)\tfrac{\partial f}{\partial y}(x,y,z,\u).\qed
\end{align*}
\end{lem}

Assume that property (i') in Lemma~\ref{analytic3vars} does not hold.
We are now ready to prove that $f$ satisfies the identity given in property (ii') of Lemma~\ref{analytic3vars}.
Define
\begin{align*}
h_1(x,y,z,\u)&:=\frac{\tfrac{\partial f}{\partial y}(x,y,z,\u)}{\tfrac{\partial f}{\partial z}(x,y,z,\u)}
\\
h_2(x,y,z,\u)&:=
\frac{\tfrac{\partial f}{\partial x}(x,y,z,\u)}{\tfrac{\partial f}{\partial z}(x,y,z,\u)} 
\\
h_3(x,y,z,\u)&:=
\frac{\tfrac{\partial f}{\partial x}(x,y,z,\u)}{\tfrac{\partial f}{\partial y}(x,y,z,\u)}
\; ;
\end{align*}
note that, by our assumption, $f$ depends non-trivially in each of its variables and hence $h_1,h_2,h_3$ are well-defined rational functions.
By Lemma~\ref{summarize}, we have
\begin{align*}
\frac{\tfrac{\partial f}{\partial y}(x,y,z,\u)}{\tfrac{\partial f}{\partial z}(x,y,z,\u)}
&\equiv\frac{\tfrac{\partial f}{\partial y}(x',y,z,\u)}{\tfrac{\partial f}{\partial z}(x',y,z,\u)}\\
\frac{\tfrac{\partial f}{\partial x}(x,y,z,\u)}{\tfrac{\partial f}{\partial z}(x,y,z,\u)} 
&\equiv
\frac{\tfrac{\partial f}{\partial x}(x,y',z,\u)}{\tfrac{\partial f}{\partial z}(x,y',z,\u)}
\\
\frac{\tfrac{\partial f}{\partial x}(x,y,z,\u)}{\tfrac{\partial f}{\partial y}(x,y,z,\u)}
&\equiv
\frac{\tfrac{\partial f}{\partial x}(x,y,z',\u)}{\tfrac{\partial f}{\partial y}(x,y,z',\u)}
.
\end{align*}
This implies that in fact
\begin{align*}
h_1(x,y,z,\u)&=h_1(y,z,\u)
\\
h_2(x,y,z,\u)&=h_2(x,z,\u)
\\
h_3(x,y,z,\u)&=h_3(x,y,\u)
.
\end{align*}
Note also that, by definition, we have
\begin{equation}\label{h123}
h_2(x,z,\u)=h_1(y,z,\u)h_3(x,y,\u),
\end{equation}
so, in particular, 
$
h_1(y,z,\u)h_3(x,y,\u)
$
is independent of $y$. Fixing some generic $y_0\in\R$, we can write
\begin{equation}\label{rp}
h_2(x,z,\u)=\frac{\tfrac{\partial f}{\partial x}(x,y,z,\u)}{\tfrac{\partial f}{\partial z}(x,y,z,\u)} =\frac{h_3(x,y_0,\u)}{1/h_1(y_0,z,\u)}=\frac{p(x,\u)}{r(z,\u)},
\end{equation}
where $p(x,\u):={h_3(x,y_0,\u)}$ and $r(z,\u):=\tfrac{1}{h_1(y_0,z,\u)}$.

In a similar manner we see that 
$$
\frac{\tfrac{\partial f}{\partial y}(x,y,z,\u)}{\tfrac{\partial f}{\partial z}(x,y,z,\u)}
=h_1(y,z,\u)=\frac{h_2(x,z,\u)}{h_3(x,y,\u)}
$$ is independent of $x$, and so, substituting $x=x_0$, we get
$$
\frac{\tfrac{\partial f}{\partial y}(x,y,z,\u)}{\tfrac{\partial f}{\partial z}(x,y,z,\u)}=\frac{q(y,\u)}{\hat r(z,\u)},
$$
where $q(y,\u):=\tfrac{1}{h_3(x_0,y,\u)}$ and $\hat r(z,\u):= \tfrac{1}{h_2(x_0,z,\u)}$.

However, by \eqref{h123}, we have $h_2(x_0,z,\u)=h_1(y_0,z,\u)h_3(x_0,y_0,\u)$, so
$$
\hat r(z,\u)=\frac{1}{h_2(x_0,z,\u)}=\frac{1}{h_3(x_0,y_0,\u)}\cdot\frac{1}{h_1(y_0,z,\u)} =\frac{1}{h_3(x_0,y_0,\u)} r(z,\u).
$$
Therefore, we can redefine 
$q(y,\u):=h_3(x_0,y_0,\u)\frac{1}{h_3(x_0,y,\u)}$ and get
\begin{equation}\label{rq}
\frac{\tfrac{\partial f}{\partial y}(x,y,z,\u)}{\tfrac{\partial f}{\partial z}(x,y,z,\u)}=\frac{q(y,\u)}{r(z,\u)}.
\end{equation}
Combining \eqref{rp} and \eqref{rq}, we get
\begin{equation}\label{pqr}
\frac{\tfrac{\partial f}{\partial x}(x,y,z,\u)}{p(x,\u)}=
\frac{\tfrac{\partial f}{\partial y}(x,y,z,\u)}{q(y,\u)}=
\frac{\tfrac{\partial f}{\partial z}(x,y,z,\u)}{r(z,\u)}
\end{equation}
for all $x,y,z$ and $\u$, where each of $p, q,r$ is a rational functions in $l+1$ variables (which is not identically zero). This proves the identity in property (ii') of Lemma~\ref{analytic3vars} and hence completes the proof of the lemma.
\hfill$\qed$

\section{Proof of Lemma~\ref{3tod}}\label{sec:3tod}
Let $f\in\R[x_1,\ldots,x_n]$ have the property from the statement. 
We may assume, without loss of generality, that, for every permutation $\sigma$ fixed, the functions $r_{\sigma,1}$, $r_{\sigma,2}$, and $r_{\sigma,3}$ do not share any irreducible component.

Fix any permutation $\sigma$, and let $x=x_{\sigma(1)}$, $y=x_{\sigma(2)}$, $z=x_{\sigma(3)}$, and $w=x_{\sigma(4)}$. If $d>4$, let $\u=(x_{\sigma(5)},\ldots,x_{\sigma(d)})$ and otherwise let $\u$ be constant.
By assumption, we have
\begin{align}\label{pqr1}
\frac
{\tfrac{\partial f}{\partial x}}
{p_1(x,w,\u)}
&=
\frac
{\tfrac{\partial f}{\partial y}}
{q_1(y,w,\u)}
=\frac
{\tfrac{\partial f}{\partial z}}
{r_1(z,w,\u)}
\\
\label{pqr2}
\frac
{\tfrac{\partial f}{\partial x}}
{p_2(x,z,\u)}
&=
\frac
{\tfrac{\partial f}{\partial y}}
{q_2(y,z,\u)}
=\frac
{\tfrac{\partial f}{\partial w}}
{r_2(w,z,\u)}
\\
\label{pqr3}
\frac
{\tfrac{\partial f}{\partial y}}
{p_3(y,x,\u)}
&=
\frac
{\tfrac{\partial f}{\partial z}}
{q_3(z,x,\u)}
=\frac
{\tfrac{\partial f}{\partial w}}
{r_3(w,x,\u)},
\end{align}
for some $(d-2)$-variate rational functions $p_i,q_i,r_i$, $i=1,2,3$.
The identities \eqref{pqr1} and \eqref{pqr2} imply that
$$
\frac
{\tfrac{\partial f}{\partial x}}{\tfrac{\partial f}{\partial y}}
=
\frac
{p_1(x,w,\u)}
{q_1(y,w,\u)}
=\frac
{p_2(x,z,\u)}
{q_2(y,z,\u)},
$$
which shows that ${\tfrac{\partial f}{\partial x}}/{\tfrac{\partial f}{\partial y}}
$ is independent of $z$ and of $w$.
Thus, we can write
\begin{align*}
p_1(x,w,\u)&=\tilde p_1(x,\u)h_1(w,\u)\\
q_1(y,w,\u)&=\tilde q_1(y,\u)h_1(w,\u) 
\end{align*}
where $h_1$ is a rational functions taken to be ``minimal'', in the sense that each (non-constant) irreducible component of $h_1$ depends non-trivially on $w$.
In a similar way, the identities \eqref{pqr1} and \eqref{pqr3} imply
$$
\frac
{\tfrac{\partial f}{\partial y}}{\tfrac{\partial f}{\partial z}}
=
\frac
{q_1(y,w,\u)}
{r_1(z,w,\u)}
=\frac
{p_3(y,x,\u)}
{q_3(z,x,\u)}
$$
or
$$
\frac
{\tfrac{\partial f}{\partial y}}{\tfrac{\partial f}{\partial z}}
=
\frac
{\tilde q_1(y,\u)h_1(w,\u)}
{r_1(z,w,\u)}
=\frac
{p_3(y,x,\u)}
{q_3(z,x,\u)},
$$
which is independent of $w$ and of $x$.
Since each (non-constant) irreducible component of $h_1$ depends non-trivially in $w$, this implies that
$$
r_1(z,w,\u)=\tilde r_1(z,\u)h_1(w,\u).
$$
Recalling our assumption that $p_1,q_1,r_1$ have no common irreducible component, we conclude that $h_1(w,\u)$ is in fact a constant.
That is, each of $p_1,q_1,r_1$ is independent of the variable $w$.

By symmetry (applying the same argument, setting $w$ to be any of $\{x_{\sigma(4)},\ldots,x_{\sigma(d)}\}$), we conclude that $p_1=p_1(x)$, $q_1=q_1(y)$, and $r_1=r_1(z)$.

Repeating the same argument, setting $x=x_1$ $y=y_1$ and $z\in\{x_3,\ldots,x_d\}$, we get
\begin{align*}
\frac
{\tfrac{\partial f}{\partial x_1}}
{r_1(x_1)}
&=
\frac
{\tfrac{\partial f}{\partial x_2}}
{r_2(x_2)}
=\frac
{\tfrac{\partial f}{\partial x_3}}
{r_3(x_3)}
\\
\frac
{\tfrac{\partial f}{\partial x_1}}
{r_{1,j}(x_1)}
&=
\frac
{\tfrac{\partial f}{\partial x_2}}
{r_{2,j}(x_2)}
=\frac
{\tfrac{\partial f}{\partial x_i}}
{\tilde r_j(x_j)},~~j=4,\ldots,d,
\end{align*}
for some univariate rational functions $r_i, r_{i,j},\tilde r_j$, for $i=1,2,3$, $j=4,\ldots,d$.
This implies that
$$
\frac{r_{1,j}(x_1)}{r_1(x_1)}=\frac{r_{j,2}(x_2)}{r_2(x_2)},
$$
which must be independent of $x_1,x_2$, and hence equals some constant $c_j$.
Finally, setting $r_j(x_j):=\tilde r_j(x_i)/c_j$, the lemma follows. \hfill $\qed$

\section{Proof of Lemma~\ref{speform}}\label{sec:speform}
We follow an argument of Tao from 
\cite[Theorem~41]{Tao}, who proved Lemma~\ref{speform} for the special case where $d=2$.
The generalization to the case of $d$ variables is straightforward, up to certain needed adjustments. 

The proof can be divided into two parts. In the first part we show 
that $f$ has an additive structure, in a sense being made precise below (Theorem \ref{additive}). 
In the second part the concrete forms stated in the lemma are deduced. 

For the first step, we generalize the following statement from \cite{Tao} to the case of $d$ variables. 

\begin{thm}[{\bf Tao~\cite[Proposition 44]{Tao}}, Additive structure in two variables]\label{tao}
Let $f\in \C[x_1,x_2]$ and assume that $$\frac
{\tfrac{\partial f}{\partial x_1}(x_1,x_2)}
{r_1(x_1)}
=
\frac
{\tfrac{\partial f}{\partial x_2}(x_1,x_2)}
{r_2(x_2)},
$$
for some univariate rational functions $r_1,r_2$.
Then there exists an entire function $H:\C\to\C$ such that
$$
f(\gamma_1(1),\gamma_2(1))=H(\int_{\gamma_1}r_1+\int_{\gamma_2}r_2),
$$
whenever $\gamma_1,\gamma_2:[0,1]\to\C$ 
are smooth curves with $\gamma_1(0)=\gamma_2(0)=0$ and 
 images not containing any pole of $r_1$ and $r_2$.
\end{thm}

Our first step is to prove a $d$-dimensional version of the above theorem. 
\begin{thm}[{Additive structure in $d$ variables}]\label{additive}
Let $f\in \C[x_1,\ldots,x_d]$ be a complex polynomial of $d$ variables.
Assume that 
\begin{equation}\label{eq1}
\frac
{\tfrac{\partial f}{\partial x_1}(x_1,\ldots,x_d)}
{r_1(x_1)}
=\cdots=
\frac
{\tfrac{\partial f}{\partial x_d}(x_1,\ldots,x_d)}
{r_d(x_d)}
\end{equation}
for some univariate rational functions $r_1,\ldots,r_d$.
Then there exists an entire function $H:\C\to\C$ such that
$$
f(\gamma_1(1),\dots,\gamma_d(1))=H(\int_{\gamma_1}r_1+\int_{\gamma_2}r_2+\dots+\int_{\gamma_d}r_d),
$$
whenever $\gamma_1,\dots,\gamma_d:[0,1]\to\C$ 
are smooth curves with $\gamma_1(0)=\gamma_2(0)=\dots=\gamma_d(0)=0$ and 
 images not containing any pole of $r_1,\dots,r_d$.
\end{thm}

We prove the theorem in the following Section \ref{prfadditive}. 
Below we shortly describe the outline of the proof. 
\paragraph{Outline of the proof of Theorem \ref{additive}.}
First, it follows from the case 
where $d=2$ that there exist a dense open subset $\Omega\subset\C^d$ and a function 
$Q:\Omega\to \C$, 
holomorphic at each of its coordinates, such that 
$$
f(\gamma_1(1),\dots,\gamma_d(1))=Q(\int_{\gamma_1}r_1,\int_{\gamma_2}r_2,\dots,\int_{\gamma_d}r_d),
$$
whenever $\gamma_1,\dots,\gamma_d:[0,1]\to\C$, are as in Theorem \ref{additive}. 
In fact, a crucial part in Tao's proof of Theorem \ref{tao} is showing the existence of such $Q$. 
Then we show that $Q$ is actually a function of 
the sum $\sum_{i=1}^d\int_{\gamma_i}r_i$, from which the existence
of $H$ follows immediately. That $Q$ is a function of the above sum follows from the fact that
its partial derivatives are all equal to each other, that is
$$
D_1 Q=D_2Q=\dots=D_dQ,
$$
which is a consequence of $\eqref{eq1}$;
here $D_iQ$ stands for the derivative of $Q$ with respect to its $i$th variable.

\subsection{Proof of Theorem \ref{additive}}\label{prfadditive}

\paragraph{Notation for the proof.}
We follow Tao's notation from \cite{Tao}. 
Fix $1\le j\le d$ and let $r=r_j$. We can write 
$$
r(x_j)=\sum_{k=1}^m\frac{\alpha_k}{x_j-a_k}+\tilde{r}(x_j),
$$
Where $a_k$ are the simple poles of $r$ with residues $\alpha_k$, and $\tilde{r}$ is a rational function 
with no simple poles. Thus $\tilde{r}$ has a primitive $R$, which is a rational function, so that
$$
r(x_j)=\sum_{k=1}^m\frac{\alpha_k}{x_j-a_k}+R'(x_j),
$$
for all but finitely many $x_i\in\C$.

By translation we may assume that $a_1,\dots,a_m\ne0$, and that $R(0)=0$.
For any smooth curve $\gamma:[0,1]\to\C$ which avoids all of the poles of $r$, and starts at $\gamma(0)=0$, 
we have
\begin{equation}\label{eq17}
\int_\gamma r=\sum_{k=1}^m\alpha_k\Log{\frac{\gamma(1)-a_k}{a_k}}+R(\gamma(1)),
\end{equation}
where (by abuse of notation) $\Log{\frac{\gamma(1)-a_k}{a_k}}$ is one of the logarithms 
$\log{\frac{\gamma(1)-a_k}{a_k}}$ of $\frac{\gamma(1)-a_k}{a_k}$. 
In particular, we have 
$$
\int_{\gamma}r\in c_{\gamma(1)}+\Gamma_j,
$$ 
where 
$$
\Gamma_j=2\pi i\alpha_1\Z+2\pi i\alpha_2\Z+\dots+2\pi i\alpha_m\Z,
$$
and for any $x\in\C$, which is not a pole of $f$, $c_x+\Gamma_j$ denotes the coset 
$$
c_x+\Gamma_j=\sum_{k=1}^m\alpha_k\log{\frac{x-a_k}{a_k}}+R(x).
$$
Thus $c_x$ is only defined up to an additive error in $\Gamma_j$. Conversely, for any given end point 
$x\in\C$, which is not a pole of $r$, and any element of the coset $z=c_x+\Gamma_j$, one can find a 
smooth curve $\gamma:[0,1]\to\C$, from $0$ to $x$, avoiding all the poles of $r$, with $\int_{\gamma}r=z$.

We will make use of the following lemma from \cite{Tao}.
\begin{lem}[{\bf Tao~\cite[Lemma 42]{Tao}}, Almost surjectivity]\label{tao2}
For all complex numbers $z$ outside of at most one coset of $\Gamma_j$, there exists at least one smooth curve 
$\gamma : [0, 1] \to \C$ starting at $0$, avoiding all the poles of $r$, 
with $\int_{\gamma} r = z$. If $\Gamma_j$ is not trivial and is not a rank one lattice $\Gamma_j = 2\pi i\alpha\Z$, then the 
caveat ``outside of at most one coset of $\Gamma_j$'' in the previous claim may be deleted.
\end{lem}

\paragraph{Proof of Theorem \ref{additive}.}

In what follows, unless stated otherwise, by a curve $\gamma$ we mean a smooth 
curve from $[0,1]$ to $\C$ that avoids all of the poles of $r_1,\dots,r_d$. 
We have the following lemma.

\begin{lem}\label{lem1}
Let $\gamma_2,\dots,\gamma_d$ be curves, starting at $0$. Then for any curves $\gamma_1,\tilde{\gamma}_1$
that start at $0$, and satisfy $\int_{\gamma_1}r_1=\int_{\tilde{\gamma}_1}r_1$, we have
$$
f(\gamma_1(1),\gamma_2(1),\dots,\gamma_d(1))=f(\tilde{\gamma}_1(1),\gamma_2(1),\dots,\gamma_d(1)).
$$
\end{lem}

\begin{proof}
Define $\tilde{f}(x_1,x_2):=f(x_1,x_2,\gamma_3(1),\dots,\gamma_d(1))$. Then $\tilde{f}$ is a polynomial in two 
variables and it satisfies \eqref{eq1}. Thus by Theorem \ref{tao} there exists a holomorphic 
function $\tilde{H}:\C\to\C$ such that
$$
\tilde{f}(\eta_1(1),\eta_2(1))=\tilde{H}(\int_{\eta_1}r_1+\int_{\eta_2}r_2),
$$ 
whenever $\eta_1$ and $\eta_2$ are curves that start at $0$. In particular 
$$
\tilde{f}(\gamma_1(1),\gamma_2(1))=\tilde{f}(\tilde{\gamma}_1(1),\gamma_2(1)).
$$
That is, 
$$
f(\gamma_1(1),\gamma_2(1),\dots,\gamma_d(1))=f(\tilde{\gamma}_1(1),\gamma_2(1),\dots,\gamma_d(1)).\qedhere
$$
\end{proof}

Let $\Omega_1$ be the set of all $z$ such that $z=\int_{\gamma_1}r_1$ 
for some curve $\gamma_1$ starting at $0$. 
Similarly define $\Omega_i$ for every $1\le i\le d$. It follows by Lemma \ref{tao2} that for each $i$,
$\Omega_i$ contains the complement of a coset of the discrete subgroup $\Gamma_i$. In particular $\Omega_i$ is open. 
By repeating the argument in the proof of Lemma \ref{lem1} for each coordinate separately, 
and since the functions $\tilde{H}$ in that proof are holomorphic, we get the following corollary.  
\begin{cor}
There exists a function $Q:\Omega_1\times\dots\times\Omega_d\to\C$, holomorphic in each of its 
coordinates,\footnote{For a function $f:\Omega_1\times\dots\times\Omega_d\to\C$, where $\Omega_1,\dots,\Omega_d$ are open subsets of $\C$, 
we say that $f$ is holomorphic in the first coordinate if for any fixed $(x_2,\dots,x_d)\in \Omega_2\times\dots\times\Omega_d$, the function
$f(\cdot,x_2,x_3,\dots,x_d)$ is holomorphic.} such that 
$$
f(\gamma_1(1),\dots,\gamma_d(1))=Q(\int_{\gamma_1}r_1,\dots,\int_{\gamma_d}r_d),
$$ 
whenever $\gamma_1,\dots,\gamma_d$ are curves starting at $0$. 
\end{cor}

We need the following version of the chain rule; the proof is technical and standard and we provide it here for completeness.
\begin{lem}\label{prop}
Let $U_1,U_2\subset\C$ be two open sets and
 $\tilde{f}:U_1\to\C$, $\tilde{Q}:U_2\to\C$  two holomorphic functions. Assume that
$$
\tilde{f}(\gamma(1))=\tilde{Q}(\int_{\gamma}r),
$$ 
for every smooth $\gamma:[0,1]\to\C$ for which expression above makes sense. That is, 
$r$ is analytic at every point in the image of $\gamma$,
$\gamma_1\in U_1$, $\int_{\gamma}\tilde{f}\in U_2$.
Then for any $\gamma$ as above
$$
\tilde{f}'(\gamma(1))=\tilde{Q}'(\int_{\gamma}r)\cdot r(\gamma(1)).
$$
\end{lem}
\begin{proof}
Let $\gamma:[0,1]\to\C$ as above, and denote $z=\int_{\gamma}r$. 
There is a small disc $D\subset U_2$ containing $\gamma(1)$ in which $r$
is holomorphic. Take a primitive $R$ of $r$ in $D$. Then $R$ is holomorphic in $D$ and $R'=r$. 
Let $w\in D$, and define $\eta:[0,1]\to\C$ to be the concatenation $\eta:=\gamma+[\gamma(1),w]$
of $\gamma$ and the interval $[\gamma(1),w]$. Then we have,  
$$
\tilde{f}(w)=\tilde{Q}(\int_{\eta}r).
$$
On the other hand, if we set $w=\gamma(1)+h$ then

\begin{align*}
\frac{\tilde{f}(\gamma(1)+h)-\tilde{f}(\gamma(1))}{h}&=\frac{\tilde{Q}(\int_{\eta}r)-\tilde{Q}(z)}{h}\\
&=\frac{\tilde{Q}(z+\int_{[\gamma(1),w]}r)-\tilde{Q}(z)}{h}\\
&=\frac{\tilde{Q}(z+R(w)-R(\gamma(1)))-\tilde{Q}(z)}{h}\\
&=\frac{\tilde{Q}(z+R(\gamma(1)+h)-R(\gamma(1)))-\tilde{Q}(z)}{h}\\
&=\frac{\tilde{Q}(z+R(\gamma(1)+h)-R(\gamma(1)))-\tilde{Q}(z)}{R(\gamma(1)+h)-R(\gamma(1))}
\cdot\frac{R(\gamma(1)+h)-R(\gamma(1)}{h}.\\ 
\end{align*}
(Note that $R(\gamma(1)+h)-R(\gamma(1))\ne0$ for small enough $0\ne h$ because 
$h\mapsto R(\gamma(1)+h)-R(\gamma(1))$ is holomorphic (and non-constant) 
in a small neighborhood of $0$, and thus have 
only a finite number of zeros in that neighborhood.)
Taking $h\to0$ we get 
$$
\tilde{f}'(\gamma(1))=\lim_{h\to0}\frac{\tilde{f}(\gamma(1)+h)-\tilde{f}(\gamma(1))}{h}=\tilde{Q}'(z)R'(\gamma(1))=\tilde{Q}'(z)r(\gamma(1)).
$$
\end{proof}

\begin{lem}\label{lem2}
Let $Q:\Omega_1\times\dots\times\Omega_d\to\C$ be a holomorphic function in each of its coordinates, 
and assume that
$$
f(\gamma_1(1),\dots,\gamma_d(1))=Q(\int_{\gamma_1}r_1,\dots,\int_{\gamma_d}r_d),
$$
whenever $\gamma_1,\dots,\gamma_d$ are curves starting at $0$. Then for every 
$\gamma(1)=(\gamma_1(1),\dots,\gamma_d(1))\in\Omega_1\times\dots\times\Omega_d$,
$$
D_1Q(\gamma(1))=\dots=D_dQ(\gamma(1)).
$$
\end{lem}

\begin{proof}
To begin, let us assume first that $r_1(\gamma_1(1)),\dots,r_d(\gamma_d(1))\ne0$.
Let $\tilde{f}=f(\cdot,\gamma_2(1),\dots,\gamma_d(1))$. Then 
$$
\tilde{f}(\gamma(1))=Q(\int_{\gamma}r_1,\int_{\gamma_2}r_2,\dots,\int_{\gamma_d}r_d),
$$
for any curve $\gamma$ that starts at $0$. By Proposition \ref{prop} we get 
$$
D_1f(\gamma_1(1),\dots,\gamma_d(1))=\tilde{f}'(\gamma_1(1))=D_1Q(\int_{\gamma_1}r_1,\dots,\int_{\gamma_d}r_d)\cdot r_1(\gamma_1(1)).
$$
Similarly, for every $1\le i\le d$
$$
D_if(\gamma_1(1),\dots,\gamma_d(1))=D_iQ(\int_{\gamma_1}r_1,\dots,\int_{\gamma_d}r_d)\cdot r_i(\gamma_i(1)).
$$
Dividing both of the sides of the above equality by $r_i(\gamma_i(1))$ we get
$$
\frac{D_if(\gamma_1(1),\dots,\gamma_d(1))}{r_i(\gamma_i(1))}=D_iQ(\int_{\gamma_1}r_1,\dots,\int_{\gamma_d}r_d).
$$
By \eqref{eq1} we get that $D_1Q=\dots=D_dQ$. 
We now show that we may  remove the assumption that $r_1(\gamma_1(1)),\dots,r_d(\gamma_d(1))\ne0$.
Let $(y_1,\dots,y_d)\in\Omega_1\times\dots\times\Omega_d$. 
We claim that there exists $\epsilon>0$ so that for every $(x_1,\dots,x_d)$ with
$0<|x_i-y_i|<\epsilon$ (for all $i$) there are 
curves $\tilde{\gamma}_1,\dots,\tilde{\gamma}_d$ so that $\tilde{\gamma}_i(0)=0$, 
$\int_{\tilde{\gamma}_i}r_i=x_i$, and $r_i(\tilde{\gamma}_i(1))\ne0$.
Given that, the argument above and the smoothness of $Q$ imply that $D_1Q=\dots=D_dQ$ at $(y_1,\dots,y_d)$.
We show that there is such $\epsilon>0$.

Let $\gamma_1$ be such that $\int_{\gamma_1}r_1=y_1$. There is a punctured disk $D$ containing $\gamma_1(1)$, so that 
$r_1$ is holomorphic in $D$, and $r_1\ne0$ in $D$. 
Let $R_1$ be a holomorphic primitive of $r_1$ in $D$. 
For every $w\in D$, we have 
$$
\int_{\gamma_1+[\gamma_1(1),w]}r_1=y_1+R_1(w)-R_1(\gamma_1(1)).
$$
Since the function on the right hand side is holomorphic in $D$, as a function of $w$, it is in particular open. 
Thus, since $y_1$ belongs to its image, there is $\epsilon_1>0$ such that any $x_1$ 
with $0<|x_1-y_1|<\epsilon_1$ belongs to its image. Thus for any such $x_1$ there is a curve 
$\tilde\gamma_1=\gamma_1+[\gamma_1(1),w]$ so that $\int_{\tilde{\gamma}_1}r_1=x_1$, $r_1(\tilde{\gamma}_1(1))\ne0$, and 
$\tilde{\gamma}_1(0)=0$.
By repeating the above argument for every $1\le i\le d$ we get $\epsilon_1,\dots, \epsilon_d$. Take 
$\epsilon=\min_i\epsilon_i$. 
\end{proof}

\begin{lem}\label{lem-additive}
Let $Q:\Omega_1\times\dots\times\Omega_d\to\C$ 
be a function holomorphic in each of its coordinates. Assume that 
$$
D_1Q=\dots=D_dQ.
$$
Then $Q(x_1,\dots,x_d)=Q(x_1',\dots,x_d')$ whenever $\sum x_i=\sum x_i'$. 
\end{lem}

\begin{proof}
Define $T:\Omega_1\times\dots\times\Omega_d\to\C^d$ by $(x_1,\dots,x_d)\mapsto(\sum_i x_i,x_2,\dots,x_d)$. 
Let $\Omega\subset\C^d$ be the image of $T$. Since $T$ is an invertible 
linear map $\Omega$ is an open subset, and we define 
$S=T^{-1}$ to be the inverse of $T$. 
$S$ is then given by $(y_1,\dots,y_d)\mapsto (y_1-\sum_{i=2}^d y_i, y_2,\dots,y_d)$. 
Let $y_1\in \C$ be in the projection of $\Omega$ onto the first coordinate. For any 
$(y_2,\dots,y_d)$ with $(y_1,y_2,\dots,y_d)\in\Omega$, define $r(y_2,\dots,y_d)=(y_1,y_2,\dots,y_d)$. 
We have
$$
(Q\circ S\circ r)'=(-D_1Q+D_2Q,-D_1Q+D_3Q,\dots,-D_1Q+D_dQ)=(0,\dots,0),
$$
and so $Q\circ S\circ r$ is constant. But this means that for any 
$(x_1,\dots,x_d),(x_1',\dots,x_d')\in\Omega_1\times\dots\times\Omega_d$, with
$\sum x_i=\sum x_i'=y_1$, we have
\begin{align*}
Q(x_1,\dots,x_d)&=
Q(y_1-\sum_{i=2}^d x_i,x_2,\dots,x_d)\\
&=Q\circ S\circ r(x_2,\dots,x_d)\\
&=Q\circ S\circ r(x_2',\dots,x_d')\\
&=Q(x_1',\dots,x_d').
\end{align*}

Since this is true for every $y_1\in\Omega_1+\dots+\Omega_d$, we get the result. 
\end{proof}

Combining Lemma  \ref{lem2} with Lemma \ref{lem-additive} we see that Theorem \ref{additive} follows. 

\subsection{Proof of Lemma \ref{speform}}

We retain the notation of the previous section.  
Let $f$ be as in the statement of the lemma. By Theorem \ref{additive}, there exists 
an entire function $H:\C\to\C$ such that
$$
f(\gamma_1(1),\dots,\gamma_d(1))=H(\int_{\gamma_1}r_1+\int_{\gamma_2}r_2+\dots+\int_{\gamma_d}r_d),
$$
whenever $\gamma_1,\dots,\gamma_d:[0,1]\to\C$ 
are smooth curves with $\gamma_1(0)=\gamma_2(0)=\dots=\gamma_d(0)=0$ and 
 images not containing any pole of $r_1,\dots,r_d$.
 
We analyze the holomorphic function $H$. 

\begin{lem}
$H$ is periodic with respect to $\Gamma_1+\dots+\Gamma_d$.
\end{lem}

\begin{proof}
Let $z\in\C$ and $\lambda\in\Gamma_1$. Take curves $\gamma_1,\dots,\gamma_d$ starting at $0$ so that
$\int_{\gamma_1}r_1+\dots+\int_{\gamma_d}r_d=z$. 
We may do this because each $\Omega_j$ has a discrete complement in $\C$. Now take $\tilde\gamma_1$ a curve starting at
$0$ so that $\tilde\gamma_1(1)=\gamma_1(1)$ and 
$\int_{\tilde\gamma_1}r_1=\int_{\gamma_1}r_1+\lambda$ (see the remark before  
Lemma \ref{tao2}). Then
$$
H(z)=f(\gamma_1(1),\dots,\gamma_d(1))=f(\tilde\gamma_1(1),\dots,\gamma_d(1))=H(z+\lambda).
$$
Thus $H$ is periodic with respect to $\Gamma_1$. Similarly it is periodic with respect to all $\Gamma_j$, and so periodic with respect
to $\Gamma_1+\dots+\Gamma_d$. 
\end{proof}

Since $H$ is periodic with respect to $\Gamma=\sum\Gamma_i$, we get that $\Gamma$ is discrete.
Indeed, otherwise, $H$ is constant on some convergent sequence, and, by the uniqueness principle, this implies that $H$ is constant on $\C$.
Thus, $f$ is constant on $\C^d$,
contradicting our assumption.

Notice also that $\Gamma$ cannot have rank $2$. Indeed, if $\Gamma$ has rank $2$ then $H$ attains its image on a compact
region in the plane, thus $H$ is bounded, thus $H$ is constant by Liouville's theorem, and we get a contradiction as before. 
Thus we may assume that either $\Gamma$ is trivial, or, without loss of generality, $\Gamma=2\pi i\Z$. We treat these two cases separately.

Suppose first that $\Gamma$ is trivial. That is, the $r_j$'s have no poles. In this case we have 
$$
f(x_1,\dots,x_d)=H(R_1(x_1)+\dots,R_d(x_d)),
$$
whenever $R_j(x_j)\ne\infty$ for each $j$. 

\begin{lem}
If $f(x_1,\dots,x_d)=H(R_1(x_1)+\dots+R_d(x_d))$ whenever $R_1(x_1)\ne\infty$, $\dots$, $R_d(x_d)\ne\infty$, then 
$H, R_1,\dots, R_d$ are polynomials. 
\end{lem}

\begin{proof}
By Rouch\'e's theorem (see Corollary \ref{rouche_cor} in the Appendix), if, say, $R_1$ has a pole at some point $a_*\in\C$, then 
 there exists a closed disk $B$, with nonempty interior containing $a_*$, in which $R_1$ takes any sufficiently large value in $\C$. That is, 
\begin{equation}\label{imageR1}
R_1(B)\supset\C\setminus \tilde B,
\end{equation}
for some closed disk $\tilde B$ centered at $0$.
Consider the function $\varphi(x)=f(x,b_2,\ldots,b_d)$ restricted to $x\in B$, where $b_2,\ldots,b_d\in\C$ are some generic fixed constants. 
Since $f$ is a polynomial, $\varphi$ is continuous in $B$ and hence bounded.
 On the other hand, we have $\varphi(x)=H(R_1(x)+\sum_{j=2}^dR_j(b_j))$, for $x\in B$. In view of \eqref{imageR1}, this implies that $H$ is bounded on 
 $$
 \C\setminus (\tilde B+\sum_{j=2}^dR_j(b_j)).
 $$
By Liouville's theorem, 
$H$ must be  constant, which yields a contradiction.
Thus, $R_1$ is a rational function with no poles, and hence it is a polynomial. Applying a symmetric argument to each $j$, we conclude that   
$R_1,\ldots,R_d$ are polynomials. 
Finally, fixing  $x_2,\ldots,x_d$ and taking $x_1$ to go to infinity, we get that $H$ has a polynomial growth, 
and thus $H$ is a polynomial, by the generalized Liouville's theorem.    
\end{proof}

Next suppose that $\Gamma=2\pi i\Z$. 
In this case we have from \eqref{eq17} that for each of the $r_j$'s,  
$$
\int_{\gamma_j} r_j=\sum_{k=1}^m\alpha_k\Log{\frac{\gamma_j(1)-a_k}{a_k}}+R_j(\gamma_j(1)),
$$
where $\alpha_k$ are integers, and $\gamma_j$ is any curve avoiding all the poles of $r_j$. In other words,
we have 
$$
\int_{\gamma_j}r_j\in R_j(\gamma_j(1))+\log\tilde{R_j}(\gamma_j(1)),
$$
where
\begin{equation}\label{tildeRj}
\tilde R_j=\prod_k\left(\frac{x-a_k}{a_k}\right)^{\alpha_k}
\end{equation}
which is a rational function. 
Note that by the assumption that $\Gamma=2\pi i\Z$, we have that 
at least one of the $\tilde{R_j}$ is non-constant.
From Theorem \ref{additive} we get 

\begin{equation}\label{eq19}
f(x_1,\dots,x_d)=H(\sum_{j=1}^dR_j(x_j)+\log(\prod_{j=1}^d\tilde{R}_j(x_j))), 
\end{equation}
whenever $R_j(x_j)\ne\infty$ and $\tilde{R}_j(x_j)\ne0,\infty$ for all $j$. 
Notice that the right-hand side is well defined since $H$ is $2\pi i\Z$-periodic. This 
periodicity also lets us define $\tilde{H}:=H\circ\log$ in $\C^*$. Since every $z \ne0$ has a 
neighborhood where an analytic branch of
$\log$ is well defined, $\tilde{H}$ is holomorphic in $\C^*$. 
Then \eqref{eq19} becomes
\begin{equation}\label{eq18}
f(x_1,\dots,x_d)=\tilde{H}(\exp(\sum_j R_j(x_j))\prod_j\tilde{R}_j(x_j)).
\end{equation}

We claim that each of the rational functions $R_1,\dots,R_d$ is in fact a polynomial; the argument
is similar to the one given above for the case $\Gamma=0$. 

Suppose for contradiction that, say, 
$R_1$, has a pole at some point $a_*\in\C$. 
Then the function 
$$
x\mapsto \exp(R_1(x)+\sum_{j=2}^dR_j(b_j)),
$$ 
where $b_2,\ldots,b_d$ are generic constants fixed, has an
essential singularity at $a_*$. Then also
$$
x\mapsto\exp(R_1(x)+\sum_{i=2}^dR_j(b_j))\tilde{R_1}(x)\prod_{j=2}^d\tilde{R_j}(b_j)
$$
has an essential singularity. By Picard's Theorem (see Theorem~\ref{picard} in the Appendix),
there exists a punctured neighborhood of $a_*$ in which the latter function can take any sufficiently large value.
Thus $\log$ of the latter function can take any sufficiently large value in a neighborhood of $a_*$, 
and so $H$ is constant by Liouville's Theorem. This contradicts our assumption that $f$ is not constant, and so 
$R_1$ is a polynomial. Similarly, $R_2,\dots,R_d$ are polynomials, as claimed.

Similarly we claim that all of the $\tilde{R}_i$'s are polynomials.  
Suppose for contradiction that, say, $\tilde{R_1}$ has a pole at some point $a_*\in\C$.
Since $\exp$ has no zeros, and all of the $R_i$'s are polynomials, we have that 
$\exp(\sum_{i=1}^dR_i(x_i))\prod_{i=1}^d\tilde{R_i}(x_i)$, being viewed as a function of $x_1$, has a pole at $a_*$. 
This again implies that $H$ is constant, which yields a contradiction. 
Thus $\tilde{R_1},\ldots,\tilde{R_d}$ are polynomials. 

Next we show that the singularity of $\tilde{H}$ at $0$ is removable. 
Assume without lost of generality that $\tilde{R_1}$ is non-constant, and let $a^*$ be a zero of it. 
Fixing $x_2,\dots,x_d$, the function 
$$
x_1\mapsto\exp(\sum_{i=1}^dR_i(x_i))\prod_{i=1}^d\tilde{R_i}(x_i)
$$ 
attains any 
small enough non-zero value in a punctured neighborhood of $a_*$, by Rouch\'e's theorem (see Corollary \ref{rouche_cor} in the Appendix).  
Thus $\tilde{H}$ is bounded in a neighborhood of $0$, which implies $0$ is a removable singularity of $\tilde H$.
So $\tilde H$ can be extended to an entire function.

Finally, we show that each of the $R_i$'s is actually constant. 
Suppose $R_1$ is non-constant, then fixing $x_2,\dots,x_d$ to be some generic values, we see from 
Rouch\'e's theorem that any large enough $z\in\C$ can be represented as 
$\exp(\sum_{i=1}^dR_i(x_i))\prod_{i=1}^d\tilde{R_i}(x_i)$ for
some $x_1=O(\log|z|)$. Thus $\tilde{H}$ grows at most like a power of a logarithm, and so by the Generalized Liouville
Theorem, $\tilde{H}$ is a constant, which yields a contradiction. Thus each of $R_1,\ldots,R_d$ is constant. 
We may then 
absorb the $\exp(\sum_{j=1}^dR_j(x_j))$ factor into $\tilde H$, and get 
\begin{equation}\label{multform}
f(x_1,\dots,x_d)=\tilde{H}(\prod_{j=1}^d\tilde{R_j}(x_j)),
\end{equation}
where $\tilde{R_j}$ are polynomials. But then $\tilde{H}$ has a polynomial growth whence by generalized Liouville theorem
$\tilde{H}$ is a polynomial. This completes the proof of Lemma~\ref{speform} for the complex case.

We now prove the real case, which is a consequence of the following lemma.
\begin{lem}
Let $f(x_1,\ldots,x_d)$ be a real polynomial. Assume that $f$ is of one the forms 
\begin{align}
f(x_1,\ldots, x_d)&=h(p_1(x_1)+\cdots+p_d(x_d))~~\text{or}\label{form1}\\
f(x_1,\ldots, x_d)&=h(p_1(x_1)\cdot\ldots\cdot p_d(x_d)),\label{form2}
\end{align}
where $h,p_1,\ldots,p_d$ are complex univariate polynomials.
Then the polynomials $h,p_1,\ldots,p_d$ can be taken to be real.
\end{lem}

\begin{proof}
We repeat the proof of Lemma~\ref{speform} with a concrete choice of the functions $r_1,\ldots,r_d$ that satisfy the differential equation \eqref{de}.
We set 
\begin{equation}\label{newrj1}
r_j:=p_j'(x_j),
\end{equation}
if $f$ has the form \eqref{form1}, and
\begin{equation}\label{newrj2}
r_j:=\frac{p_j'(x_j)}{p_j(x_j)}
\end{equation}
in case $f$ has the form \eqref{form2}.
One can easily check that for every $i,j$ 
\begin{equation}\label{quotient}
\frac{\frac{\partial f}{\partial x_i}}{\frac{\partial f}{\partial x_j}}=\frac{r_i}{r_j},
\end{equation}
and so the differential equation \eqref{de} holds in each of the cases. 
Note that since $f$ is now assumed to be real, the rational function $\frac{r_i}{r_j}$ is also real. 
Since the rational functions $r_i,r_j$ do not have a common factor, we get from \eqref{quotient} 
that $r_j=c_j\tilde r_j$, where $c_j\in\C$ and $\tilde r_j$ is a real rational function. This is true for each $j=1,\ldots,d$.
Moreover, we may take $c_1=\cdots=c_d=:c$ with this property, as is not hard  to verify.

We now repeat the proof of Lemma~\ref{speform}, where for the assumption we use  $r_1,\ldots,r_d$ from above. 
Suppose we are in the additive case, i.e., we take the $r_j$'s to be as in \eqref{newrj1}. Since in this case the $r_j$'s have no poles (and so the discrete group $\Gamma$ is trivial), we get from the proof of Lemma~\ref{speform} that $f$ has the form
$$
f(x_1,\ldots,x_d)=H(c(R_1(x_1)+\cdots+R_d(x_d))),
$$
where $R_j:=\int \tilde r_j$ and $R_1,\ldots,R_d$ are real polynomials.  
Absorbing $c$ to the function $H$, and by some abuse of notation, we get 
$$
f(x_1,\ldots,x_d)=H(R_1(x_1)+\cdots+R_d(x_d)).
$$

We claim that the polynomial $H$ can be taken to be real.
Write 
$$f(x_1,\dots,x_d)=H_1(R_1(x_1)+\dots+R_d(x_d))+iH_2(R_1(x_1)+\dots+R_d(x_d)),$$ for $H_1,H_2$ polynomials with real coefficients.
Since $f$ is real, the function $(x_1,\ldots,x_d)\mapsto H_2(R_1(x_1)+\dots+R_d(x_d))$ is zero for every $(x_1,\ldots,x_d)\in \R^d$, which implies that 
$(x_1,\ldots,x_d)\mapsto H_2(R_1(x_1)+\dots+R_d(x_d))$ is the zero function over $\C^d$ too. So we have 
$$f(x_1,\dots,x_d)=H_1(R_1(x_1)+\dots+R_d(x_d)),$$ for $H_1,R_1,\ldots,R_d$ real polynomials. This proves the additive case.

Suppose next that we are in the multiplicative case, i.e., that we take the $r_j$'s to be as in \eqref{newrj2}; in this case the $r_j$'s have some simple poles. 
Note that, for each $j$, the residues of $r_j$ are positive integers. 
Indeed, the residue of $r_j$ at, say, $a^*$ is equal by definition to  
$$
\frac1{2\pi i}\int_\gamma r_j(z)dz=\frac1{2\pi i}\int_\gamma \frac{p_j'(z)}{p_j(z)}dz,
$$ where $\gamma$ is the bundary of a small disk around $a^*$. 
Then, by the argument principle, this equals $N-P$, where $N$ is the number of zeros of $p_j$ in the disk, and $P$ is the number of poles of $p_j$.
Thus, as $p_j$ has no poles, the residues of $r_j$ are positive integers.

\begin{clm}
Let $r\in\{r_1,\ldots,r_j\}$. Suppose $a\in \C\setminus \R$ is a simple pole of $r$ with residue $\alpha$. Then $\bar{a}$ is also a simple pole of $r$ and has the same residue $\alpha$.
\end{clm}
\begin{proof}
Recall that $r=c\tilde r$, for some real rational function $\tilde r$. 
Since $r$ and $\tilde r$ have the same simple poles, and $\tilde r$ is real, $a, \bar{a}$ are simple poles of both $r$ and $\tilde r$.

Since $\tilde{r}$ is real, its residues at $a$ and 
$\bar{a}$ are conjugate to each other. Indeed, note that 
${\rm res}(\tilde{r},a)
=(z-a)\tilde r(z)\mid_{z=a}$.
Writing $\tilde{r}(z)=\frac{1}{(z-a)(z-\bar{a})}u(z)$, 
where $u$ is a rational function with real coefficients, we get
\begin{align*}
{\rm res}(\tilde{r},a)
=\frac{1}{z-\bar{a}}u(z)|_{z=a}
=\frac1{a-\bar{a}}u(a)
=\overline{\frac1{\bar{a}-a}\overline{u(a)}}
=\overline{\frac{1}{\bar{a}-a}{u(\bar{a})}}=\overline{{\rm res}(\tilde{r},\overline{a})}.
\end{align*}

Let $\alpha_1, \alpha_2$ denote the residues of $r$ corresponding to $a,\bar a$.
Then $\alpha_1=c\beta$ and $\alpha_2=c\bar\beta$, where $\beta={\rm res}(\tilde{r},a)$.
Recall that, as argued above, $\alpha_1,\alpha_2$ are positive integers, and in particular $\alpha_1, \alpha_2\in \R$. Using this, we get
$$\frac{\alpha_1}{c}=\frac{\bar{\alpha_2}}{\bar{c}}=\frac{\alpha_2}{\bar{c}}.$$
Thus $\frac{c}{\bar{c}}\in\R$ and so $c\in \R\cup i\R$. 
If $c\in i\R$ then $c=-\bar{c}$ and we get from the above equality that
$\alpha_1=-\alpha_2$. Since both $\alpha_1$ and $\alpha_2$ are positive, this is a contradiction. Thus $c\in \R$ and 
$\alpha_1=\alpha_2$ as claimed.
\end{proof}

Recall from the proof of Lemma~\ref{speform} that in the multiplicative case we get that 
$f$ is of the form
$$
f(x_1,\dots,x_d)=\tilde{H}(\prod_{j=1}^d\tilde{R_j}(x_j)),
$$
where  $$
\tilde R_j(x)=\prod_k\left(\frac{x-a_k}{a_k}\right)^{\alpha_k},$$
and where $a_k$ is a simple pole of $r_j$ with residue $\alpha_k$.
Note that the claim above implies that $\tilde R_j$ is a real polynomial for each $j$.

Now, as in \eqref{multform} we get that 
$$
f(x_1,\dots,x_d)=\tilde{H}(\prod \tilde{R_j}).
$$ 
We now need to show that $\tilde{H}$ can be taken to be with real coefficients.
The argument is similar to the one from the additive case.
Write 
$$
f(x_1,\dots,x_d)=\tilde{H_1}(\prod \tilde{R_j})+i\tilde{H_2}(\prod \tilde{R_j}),
$$ 
for $\tilde H_1,\tilde H_2$ real.
Since $f$ is real, the function $(x_1,\ldots,x_d)\mapsto \tilde{H_2}(\prod \tilde{R_j})$ must be identically zero.
So $$
f(x_1,\dots,x_d)=\tilde{H_1}(\prod \tilde{R_j}),
$$ 
with $\tilde H_1,\tilde R_1,\ldots,\tilde R_d$ real.
This proves the lemma.
\end{proof}

\appendix

\section{Properties from complex analysis theory}
In this section we provide some background in complex analysis theory, needed in Section~\ref{sec:speform}. For more details see e.g. the book~\cite{Gam}.

\paragraph{Notation.}
For any element $0\ne z\in\C$ we denote by $\log(z)$ the subset  of $\C$,
 consisting of all the solutions to the equation
$e^w=z$. That is, $\log(z)$ is the standard multiple-valued complex logarithm function.
Recall that 
$$
\log(z)=\ln|z|+i\text{Arg}(z)+2\pi i\mathbb{Z}.
$$
When fixing a specific branch of logarithm, we use the standard notation $\Log(z)$.   


\begin{thm}[Generalized Liouville's theorem]\label{liouv}
Let $f:\C\to\C$ be an entire function, and suppose that $|f(z)|\le M|z^n|$ for some $M>0$. Then  
$f$ is a polynomial of degree at most $n$. 
\end{thm}

\begin{thm}[Uniqueness principle]\label{uniq}
Let $f:\C\to\C$ be an entire function, and suppose that $f$ 
is constant on a set $S$ which has an accumulation point, then $f$ is constant.  
\end{thm}

\begin{thm}[Rouch\'e]\label{rouche}
Let $f$ and $g$ be analytic function on a simple region $D\subset\C$. If $|f(z)|>|g(z)|$ for every 
$z$ in the boundary of $D$, then $f$ and $f+g$ have the same number of zeros in $D$.
\end{thm}

\begin{cor}\label{rouche_cor}
Let $f:\C\to\C$ be a meromorphic function with a pole at $z_*$, and take $\epsilon>0$ such that the only pole of $f$ in
$B:=|z-z_*|\le\epsilon$ is $z_*$. For any large enough $w\in\C$, there is a solution $z\in B$ to the equation $f(z)=w$.
If $f$ is analytic and has a zero at $z_*$, then $f$ can take any small enough value $z\ne0$ in a neighborhood of $z_*$. 
\end{cor}

\begin{proof}
Write 
$$
f(z)=\sum_{j=1}^N\frac{a_j}{(z-z_*)^j}+g(z),
$$
where $a_N\ne0$, and $g(z)$ is a meromorphic function with no pole at $z_*$. Multiplying by $(z-z_*)^N$
we get 
$$
(z-z_*)^Nf(z)=a_N+(z-z_*)h(z), 
$$
where $h(z)$ is analytic in $B$. 
On the boundary of $B$ we have 
$$
|a_N|<|-(z-z_*)^Nw+(z-z_*)h(z)|,
$$
for every $w$ large enough.
Thus by Rouch\'e's Theorem (Theorem~\ref{rouche}), the functions $z\mapsto (z-z_*)^Nw-(z-z_*)h(z)$  and $z\mapsto a_N-(z-z_*)^Nw+(z-z_*)h(z)$ have the same number of zeros in $B$. 
In particular,  $z\mapsto a_N-(z-z_*)^Nw+(z-z_*)h(z)$ has at least one zero $z_0\in B$.  Since $a_N\ne0$ we have $z_0\neq z_*$, and so we can divide by $(z_0-z_*)^N$ to get
$$
\frac{a_N}{(z_0-z_*)^N}-w+\frac{h(z_0)}{(z_0-z_*)^{N-1}}=0.
$$
Recalling the definition of $h(z)$, this means $w=f(z_0)$.
Thus we have proved the first part of the corollary. For the second part apply the first part to the function $z\mapsto \frac{1}{f(z)}$. 
\end{proof}

\begin{thm}[Great Picard's Theorem]~\label{picard}
Suppose that a function $f:\C\to\C$ has essential singularity at a point $z_*\in\C$. Then on any 
punctured neighborhood of $z_*$, $f$ takes all possible complex values, with at most a single exception, 
infinitely often. 
\end{thm}

\end{document}